\documentclass[12pt,reqno]{amsart}
\usepackage{amscd,amssymb,amsmath,amsfonts,amsthm,latexsym}
\usepackage[english]{babel}
\usepackage{inputenc}
\usepackage{xcolor}
\newtheorem{thm}{Theorem}[section]
\newtheorem{cor}[thm]{Corollary}
\newtheorem{lem}[thm]{Lemma}
\newtheorem{prop}[thm]{Proposition}
\newtheorem{defn}[thm]{Definition}

\numberwithin{equation}{section}

\DeclareMathOperator{\Der}{Der}
\DeclareMathOperator{\Hom}{Hom}
\DeclareMathOperator{\id}{id}
\DeclareMathOperator{\Ker}{Ker}
\DeclareMathOperator{\Ima}{Im}
\DeclareMathOperator{\GL}{GL}
\DeclareMathOperator{\Mat}{Mat}

\begin{document}

\title{On the second cohomology group of simple Leibniz algebras}
\author{J.Q. Adashev}
\address{[J.Q. Adashev] Institute of Mathematics, National University of Uzbekistan, Tashkent, 100125, Uzbekistan.}
\email{adashevjq@mail.ru}
\author{M. Ladra}
\address{[M. Ladra] Department of Algebra, University of Santiago de Compostela, 15782, Spain.}
\email{manuel.ladra@usc.es}
\author{B.A. Omirov}
\address{[B.A. Omirov] Institute of Mathematics, National University of Uzbekistan, Tashkent, 100125, Uzbekistan.}
\email{omirovb@mail.ru}

\begin{abstract}
In this paper we prove some general results on Leibniz 2-cocycles for simple Leibniz algebras.
 Applying these results we establish the triviality of the second Leibniz cohomology for a simple Leibniz algebra with coefficients in itself,
  whose associated Lie algebra is isomorphic to $\mathfrak{sl}_2$.
\end{abstract}

\maketitle

\section{Introduction}

In the past years Leibniz algebras have been under active investigation for several reasons. These algebras preserve a property of Lie algebras
 that operators of right multiplication are inner derivations. Another motivation to study Leibniz algebras is that if a Leibniz algebra admits
  the additional property $[x, x]=0$ for any element of the algebra, then Leibniz identity coincides with Jacobi identity. Therefore, Lie algebras are generalized by Leibniz algebras.

From classical theory of finite-dimensional Lie algebras it is known that an arbitrary Lie algebra is decomposed into a semidirect sum
 of the solvable radical and its semisimple subalgebra (Levi's theorem) \cite{Jac}. Classification of semisimple Lie algebras has been known
  since the works of Cartan and Killing. According to the Cartan-Killing theory, a semisimple Lie algebra can be represented as a direct sum
   of simple ideals, which are completely classified.

Recently, Barnes proved an analogue of Levi's theorem for the case of Leibniz algebras \cite{Bar4}. Namely, a Leibniz algebra decomposes
 into a semidirect sum of its solvable radical and a semisimple Lie algebra. Since the semisimple part can be described from simple Lie ideals,
  one of the main problems of description of finite-dimensional Leibniz algebras consists of the study of solvable Leibniz algebras.

The inherent properties of non-Lie Leibniz algebras imply that the subspace spanned by squares of elements of the algebra is a non-trivial ideal
 (later on denoted by $I$). Moreover, the ideal $I$ is abelian and hence, it belongs to solvable radical. It is the minimal ideal with respect to the
  property of that the quotient of the algebra modulo an ideal  is a Lie algebra.

Moreover, Leibniz algebras with simple quotient by the ideal $I$ are completely described by means of simple Lie algebras and their irreducible modules.
 Practically, for such Leibniz algebras the space $I$  is a direct sum of irreducible modules.

Since in a non Lie Leibniz algebra the ideal $I$ is always non trivial, the usual notion of simplicity is not applicable for Leibniz algebras.
 Therefore, Dzhumadil'daev  in \cite{AbDz} proposed a definition of simple Leibniz algebra as an algebra which does not admit non-trivial
  ideals except $I$ and the square of an algebra is not equal to $I$.

Due to Barnes's result a simple Leibniz algebra decomposes into a  semidirect sum of simple Lie algebra and its irreducible right module.
 Later on, the notion of semisimplicity was adapted for Leibniz algebras as an algebra whose solvable radical coincides with the ideal $I$.
  These notions are in agreement with corresponding analogues in Lie algebra case.

Although semisimplicity of a Leibniz algebra $L$ is equivalent to semisimplicity of the quotient algebra of $L$ modulo the ideal $I$,
 nevertheless the result about the decomposition of a semisimple Leibniz algebra into a direct sum of simple ones is not true.

Note that Leibniz algebras of the same dimension form a variety under natural action of the linear group $\GL_n$.
 From algebraic geometry it is known that an algebraic variety is a union of irreducible components. Algebras with open orbits
  under this action in the variety of Leibniz algebras are called rigid. Closure of these open orbits provides irreducible components
   of the variety. Hence finding rigid algebras is a crucial problem from the geometrical point of view. It is well known that
    the triviality of the second group of cohomology for a Lie algebra with coefficients in itself implies rigidness of an algebra
    and the inverse statement is not true in general. Similar to the case of Lie algebras, Balavoine proved the general principles
     for deformations and rigidity of Leibniz algebras \cite{Bal}.

In this paper we investigate properties of the second group of cohomology for (semi)simple Leibniz algebras.
 Using the construction of a gradation of semisimple Leibniz algebras we induce a gradation on the space of 2-cocycles
 and prove some general results concerning homogeneous 2-cocycles. In fact, this approach allows us to use Whitehead's second lemma
  on triviality of the second group of cohomology for semisimple Lie algebras \cite{Jac}. As an application of general results
  we prove the triviality of the second group of cohomology for simple Leibniz algebras with associated Lie algebra $ \mathfrak{sl}_2$, which implies its rigidness.

\section{Preliminaries}

In this section we give some necessary definitions and preliminary results.

\begin{defn} [\cite{Lod2}] An algebra $(L,[\cdot,\cdot])$ over a field $F$ is called a
Leibniz algebra if for any $x,y,z\in L$ the so-called Leibniz identity
\[[x,[y,z]]=[[x,y],z] - [[x,z],y]\] holds.
\end{defn}

Let $L$ be a Leibniz algebra and let $I=\text{ideal} \left< [x,x]   \mid  x\in L
\right>$ be the ideal of $L$ generated by all squares. The natural
epimorphism $\varphi  \colon L \rightarrow   L/I$ determines the
associated Lie algebra $L/I$ of the Leibniz algebra $L$. It is
clear that the ideal $I$ is the minimal ideal with respect to the
property that the quotient by this ideal is a Lie algebra. In fact, the ideal $I$ coincides with the space spanned by squares of elements of an algebra \cite{Bar4}.

\begin{defn} [\cite{AbDz}]
A Leibniz algebra $L$ is called simple if its only ideals are $\{0\}, I, L$ and moreover $[L,L]\neq I$.
\end{defn}

In the next theorem we present the  multiplication table of a simple Leibniz algebra whose associated Lie algebra is $ \mathfrak{sl}_2$.

\begin{thm}[\cite{ORT}] \label{simple} Let $L$ be a complex $(m+4)$-dimensional $(m\geq 2)$ simple Leibniz algebra with associated Lie algebra
 $ \mathfrak{sl}_2$. Then it admits a basis $\{e,f,h,x_0,x_1,\dots,x_m\}$ such that non zero products of the basis vectors in $L$ are represented as follows:
\[\begin{array}{lll}
\, [x_k,e]=-k(m+1-k)x_{k-1}, & k=1, \dots, m. &\\
\, [x_k,f]=x_{k+1},  & k=0, \dots, m-1, & \\
\, [x_k,h]=(m-2k)x_k & k=0, \dots, m,&\\
\, [e,h]=2e, & [h,f]=2f, &[e,f]=h, \\
\, [h,e]=-2e& [f,h]=-2f, & [f,e]=-h.\\
 \end{array}\]
\end{thm}

For a given Leibniz algebra $L$ we define derived series as follows:
\[
L^1=L,\qquad L^{[n+1]}=[L^{[n]},L^{[n]}], \quad n \geq 1.
\]

\begin{defn} A Leibniz algebra $L$ is called
solvable if there exists an  $m\in\mathbb N$ such that $L^{[m]}=0$.
\end{defn}

Similar to the Lie algebras case, the sum of solvable ideals of Leibniz algebra is also solvable. The maximal solvable ideal is called the \emph{solvable radical} of the Leibniz algebra.

\begin{defn} A Leibniz algebra $L$ is called semisimple if its solvable radical is equal to the ideal $I$.
\end{defn}

Clearly, this definition  agrees with the definition of semisimplicity of a Lie algebra. Note that simple and a direct sum of simple Leibniz algebras
 are examples of semisimple Leibniz algebra. Unfortunately, there exist semisimple Leibniz algebras which do not decompose
  into a direct sum of simple algebras (for an example see \cite{GKO}).

The analogue of Levi's theorem from Lie theory was proved for the left Leibniz algebras, it  also holds for right Leibniz algebras (here we consider right Leibniz algebras).

\begin{thm}[Levi's Theorem \cite{Bar4}] \label{thmBarnes}
Let $L$ be a finite dimensional Leibniz algebra over a field of
characteristic zero and $R$  its solvable radical. Then there
exists a semisimple subalgebra $S$ of $L$, such that $L$ is the semidirect sum of  $S$ and $R$,
$L=S\dot{+}R$.
\end{thm}

In fact, the subalgebra $S$ in Theorem~\ref{thmBarnes} is a semisimple Lie algebra. Therefore, we have that a semisimple (simple) Leibniz algebra
is a semidirect sum of a semisimple (respectively, simple) Lie algebra $S$ and its (respectively, irreducible) right module $I$, that is, $L=S\dot{+}I$.
 Hence, we get the description of semisimple (respectively, simple) Leibniz algebras in terms of semisimple (respectively, simple) Lie algebras and their ideals $I$.

Derivations of a Leibniz algebra are defined in the usual sense.
\begin{defn}
A linear transformation $d$ of a Leibniz algebra $L$ is called a
derivation if for any $x, y\in L$,
\[d([x,y])=[d(x),y]+[x, d(y)].\]
\end{defn}
\smallskip

Consider for an arbitrary element $x\in L$ the operator of right multiplication $R_x \colon L\to L$, defined by $R_x(z)=[z,x]$. Operators of right multiplication are derivations of the Leibniz algebra $L$.



We call a vector space $M$ a {\it module over a Leibniz algebra} $L$ if there are two bilinear maps:
\[ [-,-] \colon L\times M \rightarrow M \qquad \text{and} \qquad [-,-] \colon M\times L \rightarrow M \]
satisfying the following three axioms
\begin{align*}
[m,[x,y]] & =[[m,x],y]-[[m,y],x],\\
[x,[m,y]] & =[[x,m],y]-[[x,y],m],\\
[x,[y,m]] & =[[x,y],m]-[[x,m],y],
\end{align*}
for any $m\in M$, $x, y \in L$.

For a Leibniz algebra $L$ and a module $M$ over $L$ we consider the spaces
\[CL^0(L,M) = M, \qquad CL^n(L,M)=\Hom(L^{\otimes n}, M), \ n > 0.\]

Let $d^n \colon CL^n(L,M) \rightarrow CL^{n+1}(L,M)$ be the
$F$-homomorphism defined by
 \begin{multline*}
(d^nf)(x_1, \dots , x_{n+1}): = [x_1,f(x_2,\dots,x_{n+1})]
+\sum\limits_{i=2}^{n+1}(-1)^{i}[f(x_1,
\dots, \widehat{x}_i, \dots , x_{n+1}),x_i]\\
+\sum\limits_{1\leq i<j\leq {n+1}}(-1)^{j+1}f(x_1, \dots,
x_{i-1},[x_i,x_j], x_{i+1}, \dots , \widehat{x}_j, \dots
,x_{n+1}),
\end{multline*}
 where $f\in CL^n(L,M)$ and $x_i\in L$. The property $d^{n+1}d^n=0$ leads that the derivative
operator $d=\sum\limits_{i \geq 0}d^i$ satisfies the property
$d\circ d = 0$. Therefore, the \emph{$n$-th cohomology group} is well defined by
\[HL^n(L,M): = ZL^n(L,M)/ BL^n(L,M),\]
where the elements $ZL^n(L,M):= \Ker d^{n}$ and $BL^n(L,M):=\Ima d^{n-1}$ are called {\it
$n$-cocycles} and {\it $n$-coboundaries}, respectively.

The elements $f\in BL^2(L,M)$ and $\varphi \in ZL^2(L,M)$ are
defined as follows
\begin{equation}\label{eqB2} f(x,y) = [d(x),y] + [x,d(y)] -
d([x,y]) \ \text{for some map} \ d\in \Hom(L, M), \end{equation}
\begin{equation}\label{eqZ2} [x,\varphi(y,z)]
 - [\varphi(x,y), z] + [\varphi(x,z), y] + \varphi(x, [y,z]) - \varphi([x,y],z) + \varphi([x,z],y)=0. \end{equation}

\begin{thm}[Whitehead's second lemma] \label{whitehead}  Let $G$ be a semisimple Lie algebra over a field of zero characteristic.
 Then $H^1(G,M)=H^2(G,M)=0$ for any finite-dimensional $G$-module $M$.
\end{thm}

Below, we present the cohomology version of the result on the relation of Lie homology and Leibniz homology for a Lie algebra with coefficients
 in a right module (see \cite[ Corollary 1.3]{Pir}).

Let $G$ be a Lie algebra and $M$  a right $G$-module.

\begin{thm} \label{thmpirashvili} If $H^*(G,M)=\{0\}$, then $HL^*(G,M)=\{0\}$.
\end{thm}

Later on we need the following description of the derivations of complex simple Leibniz algebras \cite{RMO}.

\begin{thm} \label{derivationsimple} Let $L=G\dot{+}I$ be a complex simple Leibniz algebra. Then any derivation $d$ of $L$
 can be represented as $d=R_a+\alpha+\triangle$, where $a\in G$, $\triangle \colon G\rightarrow I, \ \alpha = \lambda \id_{|_I}$ for some $\lambda \in \mathbb{C}$.
  In addition, if $\dim G\neq \dim I$, then $\triangle=0$. If $\dim G=\dim I$, then either $\triangle(G)=I$ or $\triangle(G)=0$.
\end{thm}

\begin{cor} \label{dimbl2sl2} In the case of $\dim G \neq \dim I$, we have $\dim BL^2(L,L)=(\dim G+\dim I)^2-\dim G -1$.

\end{cor}

Theorem~\ref{derivationsimple} can be extended for semisimple Leibniz algebras which admit a decomposition into direct sum of simple Leibniz algebras,
 that is, $L=\bigoplus_{i=1}^s L_i$, where ideals $L_i$ are simple.
\begin{thm} [\cite{RMO}] \label{derivationsemisimple} Any complex derivation $d$ of semisimple Leibniz algebra $L=\bigoplus_{i=1}^s L_i$
  can be represented as a sum of derivations of $L_i$, i.e., $d=\sum\limits_{i=1}^s d_i$, where each $d_i$ has the form as in Theorem~$\ref{derivationsimple}$.
\end{thm}

Recall, Schur's lemma which will be used later.
\begin{thm} \label{Schur} Let $G$ be a complex Lie algebra, $U$
and $V$ irreducible $G$-modules.
\begin{itemize}
\item[(i)] Any $G$-module homomorphism $\theta \colon U \rightarrow V$ is either  trivial or an isomorphism;
\item[(ii)] A linear map $\theta \colon V \rightarrow V$  is a $G$-module homomorphism if and only
if $\theta = \lambda \id_{|_V}$ for some $\lambda \in \mathbb{C}$.
\end{itemize}
\end{thm}

\section{Main Result}

\subsection{On Leibniz 2-cocycles for simple Leibniz algebras.}

\

\

Let $L$ be a $\mathbb{Z}$-graded Leibniz algebra, i.e.,  $L=\bigoplus\limits_{i\in \mathbb{Z}}L_i$ with $[L_i,L_j]\subseteq L_{i+j}$.

This gradation induces a gradation on $CL^n(L,L)$ as follows:
\[CL^n(L,L)=\bigoplus\limits_{i\in \mathbb{Z}}CL^n_{(i)},\] with
\[CL^n_{(i)} :=\{c\in CL^{n}(L,L) \mid
c(a_1,\dots,a_n)\in L_{i_1+i_2+\dots+i_n+i}, \ a_k\in L_{i_k}\}.\]

It is easy to see that the operator $d$ preserves the above gradation, that is,
\[d(CL^n_{(i)})\subseteq CL^{n+1}_{(i)}.\]

This gradation induces a $\mathbb{Z}$-gradation on
\[ZL^n(L,L)=\bigoplus\limits_{i\in \mathbb{Z}}ZL^n_{(i)},  \quad \text{and}  \quad  BL^n(L,L)=\bigoplus\limits_{i\in \mathbb{Z}}BL^n_{(i)}\]
 with $ZL^n_{(i)}=\{ c\in CL^n_{(i)} \mid d^n(c)=0\}$ and $BL^n_{(i)}=\{ d^{n-1}(c) \mid  c\in CL^{n-1}_{(i)}\}$.

According to the above $\mathbb{Z}$-gradation, we have
\[ZL^2(L,L)=\bigoplus\limits_{i\in \mathbb{Z}}ZL^2_{(i)}.\]

Let $L$ be a semisimple Leibniz algebra. Then $L=G\dot{+}I$ and by putting $L_0:=G, \ L_1:=I$ and $L_i =0$, for $i>1$ or $i < 0$, we obtain a $\mathbb{Z}$-gradation of $L$.

\begin{lem} Let $L=G\dot{+}I$ be a semisimple Leibniz algebra. Then
\[ZL^2(L,L)= ZL^2_{(-1)} \oplus ZL^2_{(0)} \oplus ZL^2_{(1)}.\]
\end{lem}
\begin{proof} It is easy to see $ZL^2_{(i)}=0$, for $i \leq -3$ and $i \geq 2$. Therefore,

\[ZL^2(L,L)= ZL^2_{(-2)} \oplus ZL^2_{(-1)} \oplus ZL^2_{(0)} \oplus ZL^2_{(1)}.\]

For $\varphi\in ZL^2(L,L)$ we have $\varphi=\varphi_{-2}+\varphi_{-1}+\varphi_0+\varphi_1$.

By definition we have $\varphi_{-2}(G,G)=\varphi_{-2}(G,I)=\varphi_{-2}(I,G)=0$ and $\varphi_{-2}(I,I)\subseteq G$.

From $(d^2\varphi_{-2})(G,I,I)=0$ we have $[G,\varphi_{-2}(I,I)]=0$. Taking into account the multiplication of  the semisimple Leibniz algebra $L$
 we conclude that $\varphi_{-2}(I,I)\subseteq I$. Thus, $\varphi_{-2}(I,I)=0$ and we get $\varphi_{-2}=0$.
\end{proof}

Let us introduce the following notations:

\begin{align*}
 \Phi_{G,I}^{-1} &=\{\varphi_{-1}\in ZL^2_{(-1)}  \mid  \varphi_{-1} \colon  G \otimes I \rightarrow
G\}, & \Phi_{G,G}^{0} & =\{\varphi_{0}\in ZL^2_{(0)}  \mid  \varphi_{0} \colon  G\otimes G \rightarrow G\},\\
 \Phi_{I,G}^{-1} &=\{\varphi_{-1}\in ZL^2_{(-1)} \mid  \varphi_{-1} \colon I \otimes G \rightarrow G\}, & \Phi_{G,I}^{0} &=\{\varphi_{0}\in ZL^2_{(0)} \mid
   \varphi_{0} \colon G\otimes I \rightarrow I\},\\
 \Phi_{I,I}^{-1} &=\{\varphi_{-1}\in ZL^2_{(-1)} \mid \varphi_{-1} \colon I\otimes I \rightarrow I\}, & \Phi_{I,G}^{0} &=\{\varphi_{0}\in ZL^2_{(0)} \mid
  \varphi_{0} \colon I\otimes G \rightarrow I\},\\
 \Phi_{G,G}^{1} &=\{\varphi_{1} \in ZL^2_{(1)} \mid  \varphi_{1} \colon G\otimes G \rightarrow I\}. &
\end{align*}

It is easy to see that
\[ZL^2(L,L)=\Phi_{G,I}^{-1}+\Phi_{I,G}^{-1}+\Phi_{I,I}^{-1}+\Phi_{G,I}^{0}+\Phi_{G,G}^{0}+\Phi_{I,G}^{0}+
\Phi_{G,G}^{1}.\]
Considering the equality \eqref{eqZ2} for homogeneous elements we deduce:

\begin{align}
&[x,\varphi(y,z)]=0, \quad x \in G, y, z \in I; \label{eq1} \\
&[x,\varphi(y,z)] +[\varphi(x,z), y] - \varphi([x,y],z)=0, \quad y \in G, x, z\in I; \label{eq2} \\
&[x,\varphi(y,z)] - [\varphi(x,y), z] + \varphi(x, [y,z]) + \varphi([x,z],y)=0, \quad z \in G, x, y\in I; \label{eq3}\\
&[x,\varphi(y,z)] + [\varphi(x,z), y] - \varphi([x,y],z)=0, \ x, y \in G, z\in I; \label{eq4}\\
& [x,\varphi(y,z)] - [\varphi(x,y), z]  + \varphi(x, [y,z]) + \varphi([x,z],y)=0, \quad x, z \in G, y \in I; \label{eq5}\\
&[x,\varphi(y,z)] - [\varphi(x,y), z] + [\varphi(x,z), y] + \varphi(x, [y,z]) - \varphi([x,y],z) + \varphi([x,z],y)=0, \ y, z \in G, x\in I; \label{eq6}\\
&[x,\varphi(y,z)] - [\varphi(x,y), z] + [\varphi(x,z), y] + \varphi(x, [y,z]) - \varphi([x,y],z) + \varphi([x,z],y)=0, \ x, y, z \in G. \label{eq7}
\end{align}

\begin{prop}  Let $L=G\dot{+}I$ be a semisimple Leibniz algebra. Then $\Phi_{G,I}^{0}=\{0\}$.
\end{prop}
\begin{proof} From equalities \eqref{eq4}--\eqref{eq5} for $\varphi_0$, we have
\begin{align*}
[\varphi_0(x,z), y]&=\varphi_0([x,y],z), &&\\
[\varphi_0(x,z), y]&= \varphi_0(x, [z,y]) + \varphi_0([x,y],z), &&   \ \text{with} \  x,y \in G, \ z\in I.
\end{align*}

Therefore, $\varphi_0(x,[z,y])=0, \, x,y \in G, \, z\in I$. Taking into account $[I,G]=I$, we conclude $\varphi_0(G,
I)=0$.
\end{proof}

For a semisimple Leibniz algebra $L=G\dot{+}I$, we consider the subspace $BL_{(0)}^2(G,G)=\{f_0\in BL_{(0)}^2(L,L) \mid \varphi_0 \colon G\otimes G \rightarrow G \}$
 of the space $BL_{(0)}^2(L,L)$.

\begin{prop}  Let $L=G\dot{+}I$ be a semisimple Leibniz algebras. Then $\Phi_{G,G}^{0}=BL_{(0)}^2(G,G)$ and $\dim \Phi_{G,G}^{0}=(\dim G)^2 - \dim G$.
\end{prop}
\begin{proof} Let us consider equality \eqref{eq7} for a $2$-cocycle $\varphi_0\in \Phi_{G,G}^{0}$ with $x, y, z\in G$ and $y=z$.
 Then we derive $[x,\varphi_0(y,y)]=0$. Hence, $\varphi_0(y,y)\in G\cap I=\{0\}$.

The equality
\[ 0=\varphi_0(x+y,x+y)=\varphi_0(x,y)+\varphi_0(y,x) \]
implies $\varphi_0(x,y)=-\varphi_0(y,x)$.

Taking into account that $\varphi_0(x,y)$ is skew-symmetric and equality \eqref{eq7}, we conclude that $\varphi_0(x,y)$
is a Lie 2-cocycle. Therefore, $\Phi_{G,G}^{0}=Z^2(G,G)$.
Due to Theorem~\ref{whitehead} we have $H^1(G,G)=H^2(G,G)=0$, which implies \[\Phi_{G,G}^{0}=Z^2(G,G)=B^2(G,G)=BL_{(0)}^2(G,G).\]

Since by definition $B^2(G,G)=\{f \in \Hom(G \wedge G,G) \mid f(x,y) = [d(x),y] + [x,d(y)] -
d([x,y]) \ \text{for some map} \ d\in \Hom(G, G)\setminus \Der(G)\}$, we get
$\Phi_{G,G}^{0}=(\dim G)^2 - \dim G$.
\end{proof}

Let us consider the subspace $BL_{1}^2(G,I):=\{f_1 \colon  G\otimes G \rightarrow I\}$ of the space $BL_{(1)}^2(L,L)$.

In next theorem we will establish the equality $\Phi_{G,G}^1=BL_{1}^2(G,I)$ and will calculate the dimension of $\Phi_{G,G}^1$.
\begin{prop} \label{prop34}
 Let $L=G\dot{+}I$ be a complex simple Leibniz algebra. Then $\Phi_{G,G}^1=BL_{1}^2(G,I)$ and
 $\dim \Phi_{G,G}^1=\dim G \cdot \dim I$ or $\dim \Phi_{G,G}^1=\dim G \cdot \dim I -1$.
\end{prop}
\begin{proof} Let $G$ be a simple Lie algebra. Then applying Theorems~\ref{whitehead} and \ref{thmpirashvili} we conclude that $HL^2(G,I)=\{0\}$, that is,
 $\Phi_{G,G}^1 = BL^2_{1}(G,I)$.

From equation \eqref{eqB2} we conclude that \[\dim BL_{1}^2(G,I)=\dim \Mat_{\dim G,\dim I}(\mathbb{C}) - \dim \Der_{(1)}(L),\]
where $\Mat_{r,s}$ is the vector space of all $r \times s$ matrices.

Recall that $d_1\in \Der_{(1)}(L)$ in Theorem~\ref{derivationsimple} is denoted by $\triangle$.

In the case $\triangle=0$, we have $\dim BL_{1}^2(G,I)=\dim G \cdot \dim I$.

Now we consider the case $\triangle\neq 0$. Note that $G$ and $I$ are irreducible $G$-modules. Moreover, $\triangle([g_1,g_2])=[\triangle(g_1),g_2]$
 for any $g_1, g_2\in G$. Applying part (i) of Theorem~\ref{Schur}, we obtain that the operator $\triangle$ is an isomorphism. Taking an appropriate basis of $I$
  we can identify $G$ and $I$ as $G$-modules. Using the second part of Theorem~\ref{Schur} we conclude $\dim \Der_{(1)}(L)=1$. Hence, $\dim BL_{1}^2(G,I)=\dim G \cdot \dim I - 1$.

\end{proof}

Thanks to Theorem~\ref{derivationsemisimple} by arguments used above we can extend the result of Proposition~\ref{prop34} to a direct sum of simple Leibniz algebras.

\begin{cor} Let $L$ be a complex semisimple Leibniz algebra that admits a decomposition into a direct sum of simple ideals, that is,
 $L=\oplus_{i=1}^s L_i$, where $L_i=G_i\dot{+}I_i$ are simple Leibniz algebras and $G=\oplus_{i=1}^s G_i$ is a semisimple Lie algebra.
  Then $\dim \Phi_{G,G}^1=\dim G \cdot \dim I$ or $\dim \Phi_{G,G}^1(G,I)=\dim G \cdot \dim I -k$, where
$k$ is the number of non-zero components in  the decomposition $\triangle=\sum_{i=1}^k \triangle_i$.
\end{cor}

\begin{prop}  Let $L=G\dot{+}I$ be a semisimple Leibniz algebra. Then
\[\dim ZL^2_{(-1)}=\dim (\Phi_{G,I}^{-1}+\Phi_{I,G}^{-1}).\]
\end{prop}
\begin{proof} Let us consider Equalities \eqref{eq2} and \eqref{eq3} for the 2-cocycle $\varphi_{-1}\in ZL^2_{(-1)}=\Phi_{I,I}^{-1}+ \Phi_{G,I}^{-1}+\Phi_{I,G}^{-1}$ with
$y \in G, \, x, z\in I$, then we get
\begin{align*}
[x,\varphi_{-1}(y,z)] +[\varphi_{-1}(x,z), y] - \varphi_{-1}([x,y],z)&=0,\\
[x,\varphi_{-1}(z,y)] - [\varphi_{-1}(x,z), y] + \varphi_{-1}(x, [z,y]) + \varphi_{-1}([x,y],z)&=0.
\end{align*}
Adding the above equalities we derive
\begin{equation} \label{eq1111}
\varphi_{-1}(x,
[z,y])=-[x,\varphi_{-1}(y,z)+\varphi_{-1}(z,y)], \ y\in G, \ x, z\in I.
\end{equation}

Thanks to equalities $[I,G]=I$ and \eqref{eq1111} we conclude that any $\varphi_{-1}'\in \Phi_{I,I}^{-1}$ is defined via
 some $\varphi_{-1}''\in \Phi_{G,I}^{-1}, \ \varphi_{-1}'''\in \Phi_{I,G}^{-1}$ and the multiplication $[I,G]$.
\end{proof}

\section{Application to simple Leibniz algebra with associated Lie algebra $\mathfrak{sl}_2$}

In this section we will consider  the simple Leibniz algebra $L= \mathfrak{sl}_2\dot{+}I$ of Theorem~\ref{simple}.

Taking into account that an arbitrary $\varphi_0\in \Phi_{G,G}^0$ is a Lie 2-cocycle, we introduce the following notations:
\[ \begin{array}{ll}
\varphi_0(e,f) = - \varphi_0(f,e) = a_{1}e+a_{2}f+a_{3}h, &\\[1mm]
\varphi_0(e,h) = - \varphi_0(h,e) = b_{1}e+b_{2}f+b_{3}h, &\\[1mm]
\varphi_0(f,h) = - \varphi_0(h,f) = c_{1}e+c_{2}f+c_{3}h, &\\[1mm]
\varphi_{0}(x_k,e) = \sum\limits_{i=0}^{m}b_{k,i}x_i, &   k=0,1,\dots,m, \\[1mm]
\varphi_{0}(x_k,f) = \sum\limits_{i=0}^{m}c_{k,i}x_i, &   k=0,1,\dots,m, \\[1mm]
\varphi_{0}(x_k,h) = \sum\limits_{i=0}^{m}d_{k,i}x_i, &   k=0,1,\dots,m. \\[1mm]
\end{array} \]

\begin{prop}  Let $L= \mathfrak{sl}_2\dot{+}I$ be the complex simple Leibniz algebra of Theorem~\ref{simple}.
Then $\dim\Phi_{I,G}^{0}=m^2+2m$.
\end{prop}
\begin{proof} Equality $(d^2\varphi_{0})(e, f, h)=0$ implies
\begin{equation}\label{eq00} a_{1}=c_{1}, \quad a_{2}=-b_{3}, \quad b_{1}=-c_{2}.\end{equation}

For $0\leq k\leq m$, we consider  the equalities:
\[(d^2\varphi_{0})(x_k, e,
f)=(d^2\varphi_{0})(x_k, e, h)=(d^2\varphi_{0})(x_k, f, h)=0.\]

Then together with restrictions \eqref{eq00} for $0\leq k,i\leq m$, we derive
\begin{align}
t_1(i)-b_{k,i-1}-(i+1)(m-i)c_{k,i+1}+d_{k,i}+k(m+1-k)c_{k-1,i}+b_{k+1,i} & =0, \label{eq11} \\
t_2(i)+2(i-k+1)b_{k,i}-(i+1)(m-i)d_{k,i+1}+k(m+1-k)d_{k-1,i} &=0,  \label{eq22} \\
t_3(i)+2(i-k-1)c_{k,i}+d_{k,i-1}-d_{k+1,i} & =0, \label{eq33}
\end{align}
where
\[t_1(i)=\left\{\begin{array}{ll}
-k(m+1-k)a_{1},& i=k-1,\\[1mm]
(m-2k)a_{3},& i=k,\\[1mm]
a_{2},& i=k+1,\\[1mm]
0,& i\neq k-1,k,k+1,\\[1mm]
\end{array}\right.\]
\[ t_2(i)=\left\{\begin{array}{ll}
-k(m+1-k)b_{1},& i=k-1,\\[1mm]
-(m-2k)a_{2},& i=k,\\[1mm]
b_{2},& i=k+1,\\[1mm]
0,& i\neq k-1,k,k+1,\\[1mm]
\end{array}\right.\]
\[ t_3(i)=\left\{\begin{array}{ll}
-k(m+1-k)a_{1},& i=k-1,\\[1mm]
(m-2k)c_{3},& i=k,\\[1mm]
-b_{1},& i=k+1,\\[1mm]
0,& i\neq k-1,k,k+1,\\[1mm]
\end{array}\right.\]
and
$b_{k,i}=c_{k,i}=d_{k,i}=t_1(i)=t_2(i)=t_3(i)=0$,  for $k,i\leq -1$ or $k,i\geq m+1$.

If $i\neq k-1$ or $i\neq k+1$, then from restrictions \eqref{eq22} and \eqref{eq33}, we find

\begin{align}
b_{k,i} & =-\frac{t_2(i)}{2(i-k+1)}+\frac{(i+1)(m-i)}{2(i-k+1)}d_{k,i+1}-\frac{k(m+1-k)}{2(i-k+1)}d_{k-1,i}, \label{eq44} \\
c_{k,i} & =\frac{t_3(i)}{2(i-k-1)}-\frac{1}{2(i-k-1)}d_{k,i-1}+\frac{1}{2(i-k-1)}d_{k+1,i}. \label{eq55}
\end{align}

If $i=k-1$ or $i=k+1$, then
from restrictions \eqref{eq22} and \eqref{eq33}, we derive

\begin{align*}
b_{1}+d_{k,k}-d_{k-1,k-1} & =0, && 1\leq k\leq m,\\
-b_{1}+d_{k,k}-d_{k+1,k+1} & =0, && 0\leq k\leq m-1.
\end{align*}

From the latter one, we obtain
\begin{equation} \label{eq66}
d_{k,k}=-kb_{1}+d_{0,0}, \ \text{for} \ 1\leq k\leq m.
\end{equation}

Consider  the restriction \eqref{eq11} for $i=k$,
\begin{equation} \label{eq111}
(m-2k)a_{1}-b_{k,k-1}-(k+1)(m-k)c_{k,k+1}+d_{k,k}+k(m+1-k)c_{k-1,k}+b_{k+1,k}=0.
\end{equation}

Summarizing equalities \eqref{eq111} for $k=0, 1, \dots, m$, we obtain
\begin{equation} \label{eq77}
d_{0,0}+d_{1,1}+\cdots+d_{m,m}=0.
\end{equation}

From \eqref{eq66} and \eqref{eq77} we deduce
\[d_{0,0}=\frac{m}{2}b_{1}.\]

Therefore, equality \eqref{eq66} can be rewritten in the following form:

\begin{equation} \label{eq88}
d_{k,k}=\frac{m-2k}{2}b_{1}, \ \ 0\leq k\leq m.
\end{equation}

Let us now calculate the number of independent parameters which define an arbitrary element of the space $\Phi_{I,G}^{0}$.

So, the number of total parameters $b_{k,i}, \, c_{k,i}, \, d_{k,i}$, with $0\leq k,i\leq m$, is equal to $3(m+1)^2$.

From restrictions \eqref{eq44} and \eqref{eq55} we find $b_{k,i},
c_{k,i}$ with $0\leq k,i\leq m$, except  $b_{k,k-1}$ and
$c_{k,k+1}$. Hence, the number of non independent parameters is
equal to $2(m+1)^2-2m$.

From equality \eqref{eq111} we conclude that the parameters $b_{k+1,k}, \ 0\leq k\leq m-1$, are not independent.
 Therefore, we have $m$ more a number of non independent parameters.

Finally, from \eqref{eq88} we get more $(m+1)$-pieces of non independent parameters.

Thus,
\[\dim \Phi_{I,G}^{0}=3(m+1)^2-\left[2(m+1)^2-2m+m+m+1\right]=m^2+2m.\]
\end{proof}

As a consequence of Proposition~\ref{prop34} we have the following proposition.

\begin{prop}  Let $L= \mathfrak{sl}_2\dot{+}I$ be a complex simple Leibniz algebra. Then $\dim\Phi_{G,G}^{1}=3(m+1)$ for $m\neq2$ and $\dim\Phi_{G,G}^{1}=3(m+1)-1$ for $m=2$.
\end{prop}

We put
\[
\begin{array}{ll}
\varphi_{-1}(e,x_k) = \alpha_{k,m+1}e+\alpha_{k,m+2}f+\alpha_{k,m+3}h, &   k=0,1,\dots,m, \\[1mm]
\varphi_{-1}(f,x_k) = \beta_{k,m+1}e+\beta_{k,m+2}f+\beta_{k,m+3}h, &   k=0,1,\dots,m, \\[1mm]
\varphi_{-1}(h,x_k) = \gamma_{k,m+1}e+\gamma_{k,m+2}f+\gamma_{k,m+3}h, &   k=0,1,\dots,m, \\[1mm]
\varphi_{-1}(x_k,e) = b_{k,m+1}e+b_{k,m+2}f+b_{k,m+3}h, &   k=0,1,\dots,m, \\[1mm]
\varphi_{-1}(x_k,f) = c_{k,m+1}e+c_{k,m+2}f+c_{k,m+3}h, &   k=0,1,\dots,m, \\[1mm]
\varphi_{-1}(x_k,h) = d_{k,m+1}e+d_{k,m+2}f+d_{k,m+3}h, &   k=0,1,\dots,m. \\[1mm]
\end{array} \]

\begin{prop} $\dim\Phi_{G,I}^{-1}=3(m+1)$.
\end{prop}
\begin{proof}
The equalities for $0 \leq i \leq m$
\[(d^2\varphi_{-1})(e, f, x_i)=(d^2\varphi_{-1})(e,
h, x_i)=(d^2\varphi_{-1})(f, h, x_i)=0\]
imply
\[\alpha_{i,m+2}=\beta_{i, m+1}=\gamma_{i, m+3}=0, \quad \alpha_{i, m+1}=-\beta_{i, m+2}, \quad \alpha_{i, m+3}=\frac{1}{2}\gamma_{i, m+2}, \quad  \beta_{i, m+3}=\frac{1}{2}\gamma_{i, m+1}.\]

Since only above equalities give us restrictions for parameters which are defined $\varphi_{-1}\in \Phi_{G,I}^{-1}$,
 we conclude that the number of independent parameters is equal to $3(m+1)$.
\end{proof}

\begin{prop} $\dim ZL^2_{(-1)}=\dim \Phi_{G,I}^{-1}$.
\end{prop}
\begin{proof} Let us consider equalities \eqref{eq4} and \eqref{eq5} for the 2-cocycle $\varphi_{-1}\in ZL^2_{(-1)}$ with
$x, y \in G, \, z\in I$. Then we get
\begin{align*}
[x,\varphi_{-1}(y,z)] + [\varphi_{-1}(x,z), y] - \varphi_{-1}([x,y],z)&=0,\\
[x,\varphi_{-1}(z,y)] - [\varphi_{-1}(x,z), y]  + \varphi_{-1}(x, [z,y]) + \varphi_{-1}([x,y],z)&=0.
\end{align*}

Summing these two equalities we derive
\begin{equation} \label{eq222}
[x,\varphi_{-1}(z,y)]=-[x, \varphi_{-1}(y,z)]- \varphi_{-1}(x,
[z,y]), \ \ \text{for} \ \ x, y\in G, z\in I.
\end{equation}

Considering equality \eqref{eq222} for the all triples $(g_1, x_i, g_2)$, where $g_1,g_2\in \{e, f, h\}$,
we obtain the relations:
\[
\begin{array}{ll}
b_{i,m+1}=\beta_{i, m+2}-\frac{i(m+1-i)}{2}\gamma_{i-1, m+1}, &   i=0,1,\dots,m, \\[1mm]
b_{i,m+2}=\frac{i(m+1-i)}{2}\gamma_{i-1, m+2}, &   i=0,1,\dots,m, \\[1mm]
b_{i,m+3}=-\frac{1}{2}\gamma_{i, m+2}-\frac{i(m+1-i)}{2}\beta_{i-1, m+2}, &   i=0,1,\dots,m, \\[1mm]
c_{i,m+1}=\frac{1}{2}\gamma_{i+1, m+1}, &   i=0,1,\dots,m, \\[1mm]
c_{i,m+2}=-\beta_{i, m+2}-\frac{1}{2}\gamma_{i+1, m+2}, &   i=0,1,\dots,m, \\[1mm]
c_{i,m+3}=-\frac{1}{2}\gamma_{i, m+1}+\frac{1}{2}\beta_{i+1, m+2}, &   i=0,1,\dots,m, \\[1mm]
d_{i,m+1}=\frac{m-2i-2}{2}\gamma_{i, m+1}, &   i=0,1,\dots,m, \\[1mm]
d_{i,m+2}=-\frac{m-2i+2}{2}\gamma_{i, m+2}, &   i=0,1,\dots,m, \\[1mm]
d_{i,m+3}=\frac{m-2i}{2}\beta_{i, m+2}, &   i=0,1,\dots,m, \\[1mm]
\end{array}\]
where $\beta_{i,m+2}=\gamma_{i,m+1}=\gamma_{i,m+2}=0$, for $i\leq
-1$ or $i\geq m+1$, which complete the proof of the  proposition.
\end{proof}

Now we formulate the main result of this section.

\begin{thm} Let $L= \mathfrak{sl}_2\dot{+}I$ be the simple Leibniz algebra of Theorem~\ref{simple}. Then $HL^2(L,L)=0$.
\end{thm}
\begin{proof} From Corollary~\ref{dimbl2sl2} and the proof of Proposition~\ref{prop34} we conclude that $\dim BL^2(L, L) =(m+4)^2-4$ for $m\neq2$ and $\dim BL^2(L, L) =(m+4)^2-5$  for $m=2$.

By above  propositions, we conclude that
\[\dim ZL^2(L,L)=\dim(\Phi_{G,I}^{-1}+\Phi_{I,G}^{-1}+\Phi_{I,I}^{-1}+
\Phi_{G,G}^{0}+\Phi_{I,G}^{0}+\Phi_{G,G}^{1})=(m+4)^2-4  \ \text{for} \  m\neq 2\]
 and
\[\dim ZL^2(L,L)=\dim(\Phi_{G,I}^{-1}+\Phi_{I,G}^{-1}+\Phi_{I,I}^{-1}+
\Phi_{G,G}^{0}+\Phi_{I,G}^{0}+\Phi_{G,G}^{1})=(m+4)^2-5 \ \text{for}  \ m=2.\]
\end{proof}

In conclusion we formulate

\textbf{Conjecture.} $HL^2(L,L)=0$ for any complex finite-dimensional simple Leibniz algebra.

\section*{Acknowledgments}

This work was partially supported by Ministerio de Econom\'ia y Competitividad (Spain),
grant MTM2013-43687-P (European FEDER support included); by Xunta de Galicia, grant GRC2013-045 (European FEDER support included) and by a grant from the Simons Foundation.


\end{document}